\documentclass{amsart}

\usepackage{hyperref}

\usepackage{amssymb}
\usepackage{amsmath}
\usepackage{amsthm}
\usepackage{enumerate}
\usepackage{graphicx}
\usepackage{tikz-cd}

\theoremstyle{plain}
\newtheorem{theorem}{Theorem}
\newtheorem{proposition}[theorem]{Proposition}
\newtheorem{lemma}[theorem]{Lemma}
\newtheorem{corollary}[theorem]{Corollary}

\newtheorem*{conj}{Conjecture}
\newtheorem{conjecture}[theorem]{Conjecture}
\newtheorem*{cor}{Corollary}

\theoremstyle{definition}
\newtheorem*{definition}{Definition}
\newtheorem*{example}{Example}

\theoremstyle{remark}

\newcommand{\RiemannSphere}{\widehat{\mathbb{C}}}

\renewcommand {\tilde} {\widetilde}

\begin{document}

\title{Removability, rigidity of circle domains and Koebe's Conjecture}

\date{November 23, 2015}

\author[M. Younsi]{Malik Younsi}
\thanks{Supported by NSERC}
\address{Department of Mathematics, Stony Brook University, Stony Brook, NY 11794-3651, United States.}
\email{malik.younsi@gmail.com}

\keywords{Circle domains, Koebe's conjecture, rigidity, removability, conformal, quasiconformal, Schottky groups.}
\subjclass[2010]{primary 30C20, 30C35; secondary 30C62, 30F40.}

\begin{abstract}
A circle domain $\Omega$ in the Riemann sphere is conformally rigid if every conformal map of $\Omega$ onto another circle domain is the restriction of a M\"{o}bius transformation. We show that two rigidity conjectures of He and Schramm are in fact equivalent, at least for a large family of circle domains. The proof follows from a result on the removability of countable unions of certain conformally removable sets. We also introduce trans-quasiconformal deformation of Schottky groups to prove that a circle domain is conformally rigid if and only if it is quasiconformally rigid, thereby providing new evidence for the aforementioned conjectures.
\end{abstract}

\maketitle

\section{Introduction}

Let $\Omega$ be a domain in the Riemann sphere $\RiemannSphere$. We say that $\Omega$ is a \textit{circle domain} if every connected component of its boundary is either a circle or a point. Circle domains appear naturally in the theory of conformal representation of planar domains. Indeed, a classical theorem of Koebe \cite{KOE1} states that any finitely connected domain is conformally equivalent to a circle domain, unique up to M\"{o}bius equivalence. This can be viewed as a generalization of the celebrated Riemann mapping theorem. Koebe's original proof was based on a continuity method using Brouwer's invariance of domain theorem, see e.g. \cite[Chapter V, Section 5]{GOL}. He later suggested several other proofs of the result, including one based on an iteration process whose convergence was finally proved by Gaier \cite{GAI}.

Koebe had earlier conjectured in 1909 that the result holds without any assumption on the connectivity of the domain. This is known as Koebe's Kreisnormierungsproblem.

\begin{conj}[Koebe \cite{KOE2}]
Any domain in $\RiemannSphere$ is conformally equivalent to a circle domain.
\end{conj}

Despite some important partial progress, the conjecture remains open. The following result of He and Schramm is undoubtedly one of the major recent advances.

\begin{theorem}[He--Schramm \cite{SCH1}]
\label{KoebeCountable}
Any domain in $\RiemannSphere$ with at most countably many boundary components is conformally equivalent to a circle domain, unique up to M\"{o}bius equivalence.
\end{theorem}

The proof is based on a transfinite induction argument and an analysis of fixed point indices of mappings. Schramm later gave in \cite{SCH3} a simpler proof using his notion of transboundary extremal length. We also mention that He and Schramm further generalized Theorem \ref{KoebeCountable} to domains whose boundary components are all circles and points except those in a countable and closed subfamily \cite{SCH4}.

The proof of Theorem \ref{KoebeCountable} relies heavily on the Baire category theorem and therefore cannot extend to the case of uncountably many boundary components. In this case, Koebe's conjecture remains wide open, although there are various families of domains with uncountably many boundary components that are known to be conformally equivalent to circle domains, see e.g. \cite{SIB2} and \cite{SCH3}.

A striking feature of He and Schramm's work is that the uniqueness part of Theorem \ref{KoebeCountable} plays a fundamental role in the proof of existence. More precisely, the latter relies on the fact that any conformal map between two circle domains with at most countably many boundary components is the restriction of a M\"{o}bius transformation. This motivates the following definition.

\begin{definition}
A circle domain $\Omega$ in $\RiemannSphere$ is said to be \textit{conformally rigid} if every conformal map of $\Omega$ onto another circle domain is the restriction of a M\"{o}bius transformation.
\end{definition}

For example, any finitely connected circle domain is rigid, as proved by Koebe himself. The proof uses reflections across the boundary circles and the fact that the limit set thereby obtained has absolute area zero and thus is removable for conformal maps on its complement. As for circle domains with countably many boundary components, they were shown to be rigid by He and Schramm \cite{SCH1}. They later proved rigidity of circle domains with $\sigma$-finite length boundaries \cite{SCH2}, which is, as far as we know, the most general rigidity result known.

On the other hand, it is well-known that circle domains with uncountably many boundary components need not be rigid. The simplest examples are complements of Cantor sets that are not conformally removable.

\begin{definition}
Let $E \subset \mathbb{C}$ be compact. We say that $E$ is \textit{conformally removable} if every homeomorphism of $\RiemannSphere$ which is conformal outside $E$ is actually conformal everywhere, hence is a M\"{o}bius transformation.
\end{definition}

It follows from the measurable Riemann mapping theorem (cf. Theorem \ref{MRMT}) that quasicircles are conformally removable whereas sets of positive area are not. Moreover, compact sets of $\sigma$-finite length are well-known to be conformally removable. These results are best possible from the point of view of Hausdorff measure, since there exist removable sets of Hausdorff dimension two and non-removable sets of Hausdorff dimension one. We also mention that there has been much recent interest in conformal removability, mainly because of its applications to holomorphic dynamics and conformal welding. For more information, the reader may consult the survey article \cite{YOU}.

The failure of rigidity for some circle domains appears to be one of the main difficulties in the study of Koebe uniformization. It is believed that the precise understanding of rigidity should provide substantial insight into Koebe's conjecture. Motivated by this, He and Schramm proposed the following characterization.

\begin{conjecture}[Rigidity Conjecture \cite{SCH2}]
\label{RigidityConjecture}
Let $\Omega$ be a circle domain. Then the following are equivalent :
\begin{description}
\item[(A)] $\Omega$ is conformally rigid;
\item[(B)] the boundary of $\Omega$ is conformally removable;
\item[(C)] any Cantor set contained in the boundary of $\Omega$ is conformally removable.
\end{description}
\end{conjecture}

Note that if $\Omega$ is the complement of a Cantor set $E$, then $\Omega$ is a circle domain and $E$ is conformally removable whenever $\Omega$ is rigid.

The present paper is devoted to the study of Conjecture \ref{RigidityConjecture}. Our first main result is the following.

\begin{theorem}
\label{MainThm1}
Let $\Omega$ be a circle domain whose boundary is the union of countably many circles, Cantor sets and singletons. Then \textbf{\emph{(B)}} and \textbf{\emph{(C)}} are equivalent.
\end{theorem}

Note that there can be at most countably many circles in the boundary of a circle domain. The assumption that $\partial \Omega$ is the union of countably many circles, Cantor sets and singletons holds whenever the set of point boundary components which are accumulation points of circles is an $F_\sigma$ set. In particular, this occurs if there are only finitely many circles or, more generally, if all point boundary components are isolated from the circles.

We also give an example showing that the condition fails for some circle domains. We do not know whether \textbf{(B)} and \textbf{(C)} are equivalent in this case.

Theorem \ref{MainThm1} is a direct consequence of the following result on unions of certain removable sets, which is of independent interest.

\begin{theorem}
\label{RemovableUnion}
Let $E$ be a compact plane set of the form
$$E=\bigcup_{j=1}^\infty \Gamma_j \cup \bigcup_{k=1}^\infty C_k \cup \bigcup_{l=1}^\infty \{z_l\},$$
where each $\Gamma_j$ is a quasicircle, each $C_k$ is a Cantor set and each $z_l$ is a complex number. Then $E$ is conformally removable if and only if every $C_k$ is conformally removable.
\end{theorem}
Note that it is open in general whether the union of two removable sets is removable. As far as we know, Theorem \ref{RemovableUnion} is the first non-trivial result on unions of removable sets.

Our second main result is related to the equivalence of \textbf{(A)} and \textbf{(B)}. It is well-known that in the definition of a conformally removable set, conformal maps can be replaced by quasiconformal mappings; in other words, the notions of conformal removability and quasiconformal removability actually coincide (cf. Proposition \ref{CHremQCHrem}). Therefore, if \textbf{(A)} is equivalent to \textbf{(B)}, then conformal rigidity should in principle be equivalent to quasiconformal rigidity. We show that it is indeed the case.

\begin{theorem}
\label{MainThm2}
A circle domain $\Omega$ is conformally rigid if and only if it is quasiconformally rigid.
\end{theorem}

The proof requires a new technique involving David maps, which we refer to as trans-quasiconformal deformation of Schottky groups. It is quite reminiscent of trans-quasiconformal surgery in holomorphic dynamics (see e.g. \cite[Chapter 9]{BRF}).

An immediate consequence of Theorem \ref{MainThm2} is that rigid circle domains are invariant under quasiconformal mappings of the sphere.

\begin{corollary}
\label{InvQCcircle}
Let $\Omega$ be a circle domain and let $f$ be a quasiconformal mapping of the sphere which maps $\Omega$ onto another circle domain $f(\Omega)$. If $\Omega$ is conformally rigid, then $f(\Omega)$ is also conformally rigid.
\end{corollary}

Since conformally removable sets are also quasiconformally invariant (cf. Corollary \ref{InvQC}), this provides more evidence in favor of Conjecture \ref{RigidityConjecture}.

The paper is structured as follows. Section \ref{Sec1} contains some preliminaries on quasiconformal mappings and conformally removable sets. The main result of this section is a conformal extension result which is needed for the proof of Theorem \ref{RemovableUnion}, detailed in Section \ref{Sec2}. In Section \ref{Sec3}, we introduce David maps and trans-quasiconformal deformation of Schottky groups. Section \ref{Sec4} contains the proof of Theorem \ref{MainThm2}. Finally, we conclude the paper with Section \ref{Sec5} and the discussion of a possible approach to the implication \textbf{(B)} $\Rightarrow$ \textbf{(A)} in Conjecture \ref{RigidityConjecture}.

\section{Preliminaries and conformal extension}
\label{Sec1}
The first part of this section consists of a very brief introduction to quasiconformal mappings. For more information, we refer the reader to \cite{LEH}.

First, we start with the analytic definition of quasiconformal mappings.

\begin{definition}
Let $K \geq 1$, let $U,V$ be domains in $\RiemannSphere$ and let $f:U \to V$ be an orientation-preserving homeomorphism. We say that $f$ is $K$-\textit{quasiconformal} if it belongs to the Sobolev space $W_{loc}^{1,2}(U)$ and satisfies the Beltrami equation
$$\partial_{\overline{z}}f=\mu \, \partial_z f$$
almost everywhere on $U$, for some measurable function $\mu : U \to \mathbb{D}$ with $\|\mu\|_\infty \leq \frac{K-1}{K+1}$. In this case, the function $\mu$ is called the \textit{Beltrami coefficient} of $f$ and is denoted by $\mu_f$.
\end{definition}

A mapping is conformal if and only if it is $1$-quasiconformal. This is usually referred to as Weyl's lemma. Furthermore, inverses of $K$-quasiconformal mappings are also $K$-quasiconformal, and the composition of a $K_1$-quasiconformal mapping and a $K_2$-quasiconformal mapping is $K_1K_2$-quasiconformal. We will also need the fact that quasiconformal mappings preserve sets of area zero, in the sense that if $E \subset U$ is measurable, then $m(E)=0$ if and only if $m(f(E))=0$, where $m$ is the two-dimensional Lebesgue measure.

The following theorem is of central importance in the theory of quasiconformal mappings.

\begin{theorem}[Measurable Riemann mapping theorem]
\label{MRMT}
Let $U$ be a domain in $\RiemannSphere$ and let $\mu:U \to \mathbb{D}$ be a measurable function with $\|\mu\|_\infty <1$. Then there exists a quasiconformal mapping $f$ on $U$ such that $\mu=\mu_f$, i.e.
$$\partial_{\overline{z}}f=\mu \, \partial_z f$$
almost everywhere on $U$. Moreover, the map $f$ is unique up to post-composition by a conformal map, in the sense that a quasiconformal mapping $g$ on $U$ satisfies $\mu_g=\mu=\mu_f$ if and only if $f \circ g^{-1} : g(U) \to f(U)$ is conformal.
\end{theorem}

The second part of this section deals with the properties of conformally removable sets that are needed for the proof of Theorem \ref{RemovableUnion}. First, we introduce the following definition.

\begin{definition}
Let $E \subset \mathbb{C}$ be compact. We say that $E$ is \textit{quasiconformally removable} if every homeomorphism of $\RiemannSphere$ which is quasiconformal outside $E$ is actually quasiconformal everywhere.
\end{definition}

Standard quasiconformal extension results (see e.g. \cite[Chapter II, Theorem 8.3]{LEH}) imply that quasiconformal removability is a local property.

\begin{proposition}
\label{QCextend}
Let $E \subset \mathbb{C}$ be compact. Then the following are equivalent :

\begin{enumerate}[\rm(i)]
\item For any open set $U$ with $E \subset U$, every homeomorphism $f:U \to f(U)$ which is quasiconformal on $U \setminus E$ is actually quasiconformal on the whole open set $U$;
\item $E$ is quasiconformally removable.
\end{enumerate}
\end{proposition}

A remarkable consequence of the measurable Riemann mapping theorem is that the notions of quasiconformal removability and conformal removability actually coincide.

\begin{proposition}
\label{CHremQCHrem}
A compact set $E \subset \mathbb{C}$ is quasiconformally removable if and only if it is conformally removable.
\end{proposition}

\begin{proof}
Assume that $E$ is quasiconformally removable. First, we note that the area of $E$ must be zero. Indeed, this is a direct consequence of a result of Kaufman and Wu \cite{KAU} which states that if $E$ has positive area, then there is a homeomorphism of $\RiemannSphere$ which is conformal outside $E$ and maps a subset of $E$ of positive area onto a set of zero area. Now, let $f:\RiemannSphere \to \RiemannSphere$ be a homeomorphism which is conformal on the complement of $E$. Then in particular $f$ is quasiconformal outside $E$, so it must be quasiconformal on the whole sphere, by quasiconformal removability of $E$. But then it follows from Weyl's lemma that $f$ is a M\"{o}bius transformation. Thus $E$ is conformally removable.

Conversely, assume that $E$ is conformally removable and let $g:\RiemannSphere \to \RiemannSphere$ be any homeomorphism which is quasiconformal on $\RiemannSphere \setminus E$. By Theorem \ref{MRMT}, there exists a quasiconformal mapping $f:\RiemannSphere \to \RiemannSphere$ such that $f \circ g$ is conformal on $\RiemannSphere \setminus E$. Since $E$ is conformally removable, the map $f \circ g$ is a M\"{o}bius transformation and thus $g=f^{-1} \circ (f \circ g)$ is actually quasiconformal on the whole sphere. It follows that $E$ is quasiconformally removable.

\end{proof}

Note that as a corollary of Proposition \ref{QCextend} and Proposition \ref{CHremQCHrem}, we obtain that the union of two disjoint conformally removable sets is conformally removable. Another important consequence is that the notion of conformal removability is quasiconformally invariant.

\begin{corollary}
\label{InvQC}
Let $f$ be a quasiconformal mapping of the sphere with $f(\infty)=\infty$. If $E$ is conformally removable, then $f(E)$ is also conformally removable.
\end{corollary}

In particular, every quasicircle (image of the unit circle $\mathbb{T}$ under a quasiconformal mapping of the sphere) is conformally removable.

We shall also need the following analogue of Proposition \ref{QCextend}.

\begin{proposition}
\label{Cextend}
Let $E \subset \mathbb{C}$ be compact. Then the following are equivalent :

\begin{enumerate}[\rm(i)]
\item For any open set $U$ with $E \subset U$, every homeomorphism $f:U \to f(U)$ which is conformal on $U \setminus E$ is actually conformal on the whole open set $U$;
\item $E$ is conformally removable.
\end{enumerate}

\end{proposition}

\begin{proof}
The fact that (i) implies (ii) is trivial. For the converse, assume that $E$ is conformally removable, and let $U$ be any open set with $E \subset U$ and $f:U \to f(U)$ be any homeomorphism which is conformal on $U \setminus E$. By Proposition \ref{CHremQCHrem}, the set $E$ is quasiconformally removable and therefore $f$ is quasiconformal on the whole open set $U$, by Proposition \ref{QCextend}. Since $E$ has zero area, the map $f$ is actually conformal on $U$, by Weyl's lemma.

\end{proof}

For the proof of Theorem \ref{RemovableUnion}, we need a version of Proposition \ref{Cextend} which holds without the assumption that $U$ contains $E$ in (i). We do not know if this is true in general (it would imply that the union of two conformally removable sets is conformally removable even if the sets are not disjoint). However, it is true with some additional hypothesis on the compact set $E$. This is the main result of this section.

\begin{theorem}
\label{ThmCextend}
Let $E \subset \mathbb{C}$ be compact, and assume that $E$ is either a quasicircle or a totally disconnected conformally removable set. Then for any open set $U$, every homeomorphism $f:U \to f(U)$ which is conformal on $U \setminus E$ is actually conformal on the whole open set $U$.
\end{theorem}

For the proof, we need the following elementary lemma from point-set topology.

\begin{lemma}
\label{LemmaTotDisc}
Let $S$ be a totally disconnected compact Hausdorff space. Suppose that $C_1$ and $C_2$ are two disjoint closed subsets of $S$. Then there exist disjoint closed subsets $S_1$ and $S_2$ of $S$ such that $S=S_1 \cup S_2$, $C_1 \subset S_1$ and $C_2 \subset S_2$.
\end{lemma}

\begin{proof}
The assumptions on $S$ imply that it is zero-dimensional, i.e. it has a basis consisting of clopen sets. It follows that for any $x \in C_1$, there is a clopen set $O_x \subset S \setminus C_2$ with $x \in O_x$. Since $C_1$ is compact, there exist $x_1,\dots,x_n \in C_1$ such that $C_1 \subset \cup_{j=1}^n O_{x_j} :=S_1$. The result then follows by setting $S_2:= S \setminus S_1$.
\end{proof}

We can now prove Theorem \ref{ThmCextend}.

\begin{proof}
Let $U$ be open and let $f:U \to f(U)$ be a homeomorphism which is conformal on $U \setminus E$.

Assume first that $E$ is a quasicircle. In this case, the result follows from a simple application of Morera's theorem and the measurable Riemann mapping theorem. Write $E=F(\mathbb{T})$, where $F: \RiemannSphere \to \RiemannSphere$ is quasiconformal. By the measurable Riemann mapping theorem, there is a quasiconformal mapping $g:\RiemannSphere \to \RiemannSphere$ such that
$$\mu_{F^{-1} \circ f^{-1}} = \mu_g$$
on $f(U \setminus E)$. Then $g \circ f \circ F$ is a homeomorphism on $F^{-1}(U)$ which is conformal on $F^{-1}(U) \setminus F^{-1}(E)$. But $F^{-1}(E)=\mathbb{T}$, so that $g \circ f \circ F$ is conformal everywhere on $F^{-1}(U)$.

Indeed, let $z_0$ be any point in $F^{-1}(U)$ and let $r>0$ such that $\mathbb{D}(z_0,r) \subset F^{-1}(U)$. Then $g \circ f \circ F$ is continuous on $\mathbb{D}(z_0,r)$ and holomorphic everywhere on that disk except maybe on a circular arc. By Morera's theorem, the map $g \circ f \circ F$ is holomorphic everywhere on $\mathbb{D}(z_0,r)$.

This implies that $f$ is quasiconformal on $U$. Since $E$ is a quasicircle, its area must be zero, so that by Weyl's lemma, the map $f$ is conformal everywhere on $U$. This completes the proof in the case where $E$ is a quasicircle.

Suppose now that $E$ is totally disconnected and conformally removable. Fix $z_0 \in U$ and let $r>0$ such that $\overline{\mathbb{D}}(z_0,r) \subset U$. Define
$$C_1:= \{z \in E : \operatorname{dist}(z,\mathbb{C} \setminus U) \geq \epsilon\}$$
and $C_2:=E \setminus U$, where $\epsilon>0$ is smaller than the distance between $\overline{\mathbb{D}}(z_0,r)$ and the complement of $U$. Then $C_1$ and $C_2$ are two disjoint closed subsets of $E$, so by Lemma \ref{LemmaTotDisc} there exist two disjoint closed subsets of $E$, say $E_1$ and $E_2$, such that $E=E_1 \cup E_2$, $C_1 \subset E_1$ and $C_2 \subset E_2$. Then $E_1$ is conformally removable. Moreover, since $f$ is conformal on $(U \setminus E_2) \setminus E_1$ and $E_1$ is a compact subset of the open set $U \setminus E_2$, it follows from Proposition \ref{Cextend} that $f$ is conformal on $U \setminus E_2$, a set containing the disk $\mathbb{D}(z_0,r)$. As $z_0 \in U$ was arbitrary, we get that $f$ is conformal everywhere on $U$.
\end{proof}

\section{Proof of Theorem \ref{RemovableUnion}}
\label{Sec2}
We now have the required preliminaries for the proof of Theorem \ref{RemovableUnion}.

Let $E$ be a compact plane set of the form
$$E=\bigcup_{j=1}^\infty \Gamma_j \cup \bigcup_{k=1}^\infty C_k \cup \bigcup_{l=1}^\infty \{z_l\},$$
where each $\Gamma_j$ is a quasicircle, each $C_k$ is a Cantor set and each $z_l$ is a complex number. We want to show that $E$ is conformally removable if and only if every $C_k$ is conformally removable.

\begin{proof}
If $E$ is conformally removable, then clearly every $C_k$, as a compact subset of $E$, is also conformally removable.

Conversely, assume that every $C_k$ is conformally removable. Let $f:\RiemannSphere \to \RiemannSphere$ be a homeomorphism which is conformal on $D:= \RiemannSphere \setminus E$. Let $D'$ be the union of all open subsets of $\RiemannSphere$ on which $f$ is conformal, so that $D \subset D'$. Let us prove that $D'=\RiemannSphere$, i.e. $E':=\RiemannSphere \setminus D'=\emptyset$. Assume for a contradiction that $E'$ is not empty. Since $E'$ is the countable union of the closed sets $\Gamma_j \cap E'$, $C_k \cap E'$ and $\{z_l\} \cap E'$, one of these sets must have nonempty interior in $E'$, by the Baire category theorem. Suppose that there is some $j \in \mathbb{N}$ such that $\Gamma_j \cap E'$ has nonempty interior in $E'$. The argument is the same if $C_k \cap E'$ has nonempty interior in $E'$ for some $k$ or if $\{z_l\} \cap E'$ has nonempty interior in $E'$ for some $l$. Then there is an open set $U \subset \mathbb{C}$ such that $U \cap E' \neq \emptyset$ and $U \cap E' \subset \Gamma_j \cap E'$. Since $f:\RiemannSphere \to \RiemannSphere$ is a homeomorphism which is conformal on $U \setminus E' = U \setminus (\Gamma_j \cap E') \supset U \setminus \Gamma_j$ and $\Gamma_j$ is a quasicircle, we obtain from Theorem \ref{ThmCextend} that $f$ is conformal on $U$. Hence $f$ is conformal on $D' \cup U$, an open set properly containing $D'$, contradicting the maximality of the latter. Therefore $D'=\RiemannSphere$ and $f$ is conformal on the whole Riemann sphere. This shows that $E$ is conformally removable.

\end{proof}

\begin{cor}
Let $\Omega$ be a circle domain whose boundary is the union of countably many circles, Cantor sets and singletons. Then the boundary of $\Omega$ is conformally removable if and only if any Cantor set contained in the boundary of $\Omega$ is conformally removable. In other words, \textbf{\emph{(B)}} and \textbf{\emph{(C)}} in Conjecture \ref{RigidityConjecture} are equivalent.
\end{cor}

\begin{corollary}
Let $\Omega$ be a circle domain and let $P$ be the set of its point boundary components. Suppose that the subset $A$ of $P$ consisting of accumulation points of circles is an $F_\sigma$. Then \textbf{\emph{(B)}} and \textbf{\emph{(C)}} in Conjecture \ref{RigidityConjecture} are equivalent.
\end{corollary}

\begin{proof}
If $A$ is the countable union of the closed sets $F_k$, $k \in \mathbb{N}$, then we can write $\partial \Omega$ as a countable union of closed sets
$$\partial \Omega = \bigcup_{j=1}^\infty \gamma_j \cup \bigcup_{k=1}^\infty F_k \cup \bigcup_{n=1}^\infty P_n,$$
where each $\gamma_j$ is a circle and
$$P_n:= \{z \in P : \operatorname{dist}(z,\gamma_j) \geq 1/n\, \, \mbox{for} \, \, j=1,2,\dots \} \qquad (n \in \mathbb{N}).$$
By the Cantor--Bendixson theorem, each $F_k$ and each $P_n$ can be written as the union of a Cantor set and a set which is at most countable. The result then follows from the previous corollary.

\end{proof}

We conclude this section with an example of a circle domain $\Omega$ whose boundary is not the union of countably many circles, Cantor sets and singletons.

\begin{example}
It is not difficult to construct inductively a circle domain $\Omega$ such that every $z \in \partial \Omega$ is an accumulation point of infinitely many circles. For example, consider the following construction, due to Lasse Rempe-Gillen.

At each inductive step $k$, we construct a finite collection $\mathcal{F}_k$ of pairwise disjoint circles, with $\mathcal{F}_k \subset \mathcal{F}_{k+1}$. Moreover, for each circle $\gamma \in \mathcal{F}_k$, we pick an open annulus $A_k(\gamma)$ surrounding $\gamma$ and contained in a $1/k$-neighborhood of $\gamma$, which separates $\gamma$ from the other circles of $\mathcal{F}_k$.

Start with $\mathcal{F}_1$ containing only one circle $\gamma_1$, with some annulus $A_1(\gamma)$ surrounding it. Then, inductively, for every circle $\gamma \in \mathcal{F}_k$, add small circles to $\mathcal{F}_{k+1}$, each within the inner curve of $A_k(\gamma)$, such that every point of $\gamma$ is within a distance of $1/k$ from one of these new circles. Then pick the annuli with the desired properties.

Let $K$ be the closure of the union of all the circles in the $\mathcal{F}_k$'s. Then it is easy to see that every component of $K$ is either a circle or a point (two distinct points in $K$ either belong to the same circle or are separated by some annulus). Letting $\Omega$ be the unbounded component of $\RiemannSphere \setminus K$ gives a circle domain with the required properties.

\begin{figure}[h!t!b]
\label{fig1}
\begin{center}
\includegraphics[width=6cm, height=6cm]{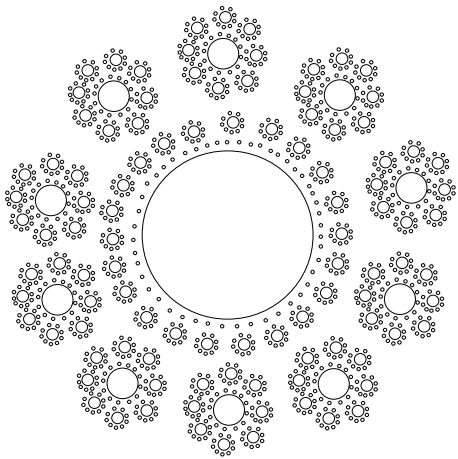}
  \caption{The compact set $K$.}
\end{center}
\end{figure}

Now, note that the boundary of such a circle domain $\Omega$ cannot be the union of countably many circles, Cantor sets and singletons. Indeed, assume for a contradiction that it is. Then one of these countably many circles, Cantor sets and singletons, say $F$, would have non-empty interior in $\partial \Omega$, by the Baire category theorem. It follows that there is an open set $U$ with $U \cap \partial \Omega \neq \emptyset$ and $U \cap \partial \Omega \subset F$. But then any point in $U \cap \partial \Omega$ is not an accumulation point of infinitely many circles, a contradiction.

\end{example}

\section{Trans-quasiconformal deformation of Schottky groups}
\label{Sec3}
Let $\Omega$ be a circle domain and let $\{\gamma_j\}_{j=1}^\infty$ be the collection of disjoint circles in $\partial \Omega$. For notational convenience, we assume that there are infinitely many boundary circles, although of course everything also works in the finite case. For $j \in \mathbb{N}$, denote by $R_j : \RiemannSphere \to \RiemannSphere$ the reflection across the circle $\gamma_j$ :
$$R_j(z)=a_j + \frac{r_j^2}{\overline{z-a_j}},$$
where $a_j$ is the center and $r_j$ is the radius of the circle $\gamma_j$.

\begin{definition}
The \textit{Schottky group} $\Gamma(\Omega)$ is the free discrete group of M\"{o}bius and anti-M\"{o}bius transformations generated by the family of reflections $\{R_j\}_{j \in \mathbb{N}}$.
\end{definition}

Thus $\Gamma(\Omega)$ consists of the identity map and all transformations of the form $R_{i_1}\circ \cdots \circ R_{i_k}$ where $k \in \mathbb{N}$, $i_1,\dots,i_k \in \mathbb{N}$ and $i_j \neq i_{j+1}$ for $j=1,\dots,k-1$.

Classical quasiconformal deformation of Schottky groups deals with Beltrami coefficients which are invariant under $\Gamma(\Omega)$. As far as we know, it was first introduced by Sibner \cite{SIB1} to show that a domain in $\RiemannSphere$ is conformally equivalent to a circle domain if and only if it is quasiconformally equivalent to a circle domain (cf. Proposition \ref{QCeqCeq}). See also \cite[Section 2]{SCH4}. Lastly, we mention that the method was used recently by Bonk, Kleiner and Merenkov to study quasisymmetric rigidity of Schottky sets, see \cite[Section 7]{BON}.

In this section, we introduce a generalization of classical quasiconformal deformation where Beltrami coefficients are replaced by so-called David coefficients. These are measurable functions bounded by one that are allowed to tend to one in a controlled way.

\begin{definition}
Let $U \subset \mathbb{C}$ be open. We say that a measurable function $\mu:U \to \mathbb{D}$ is a \textit{David coefficient} if there exist constants $M>0$, $\alpha>0$ and $0<\epsilon_0<1$ such that
\begin{equation}
\label{DavidCondition}
m(\{z \in U : |\mu(z)|>1-\epsilon\}) < M e^{-\frac{\alpha}{\epsilon}} \qquad (\epsilon<\epsilon_0),
\end{equation}
where $m$ is the two-dimensional Lebesgue measure.

Furthermore, an orientation-preserving homeomorphism $f$ on $U$ is called a \textit{David map} if $f$ belongs to the Sobolev space $W_{loc}^{1,1}(U)$ and satisfies the Beltrami equation
$$\partial_{\overline{z}}f=\mu \, \partial_z f$$
almost everywhere on $U$, for some measurable function $\mu : U \to \mathbb{D}$ satisfying (\ref{DavidCondition}). In this case, the function $\mu$ is called the \textit{David coefficient of} $f$ and is denoted by $\mu_f$.

\end{definition}

David coefficients and David maps were introduced by David \cite{DAV} for the study of the Beltrami equation in the degenerate case $\|\mu\|_\infty = 1$. More recently, they appeared to be quite useful in holomorphic dynamics, see e.g. \cite[Chapter 9]{BRF}.

It is well-known that David maps share some of the useful properties of quasiconformal mappings. For instance, they preserve sets of area zero in the sense that for any measurable set $E \subset U$, we have $m(E)=0$ if and only if $m(f(E))=0$. On the other hand, inverses of David maps may not be David; this is quite problematic in many situations. To circumvent this difficulty, one can replace (\ref{DavidCondition}) by the stronger condition
\begin{equation}
\label{StronglyDavidCondition}
m(\{z \in U : |\mu(z)|>1-\epsilon \}) \leq M e^{-\beta e^{\frac{\alpha}{\epsilon}}} \qquad (\epsilon>0)
\end{equation}
for some constants $M, \alpha, \beta>0$. We shall denote the corresponding measurable functions $\mu:U \to \mathbb{D}$ by \textit{strongly David coefficients} and the corresponding homeomorphisms $f$ on $U$ by \textit{strongly David maps}. Then inverses and compositions of strongly David maps are strongly David, see \cite[Section 11]{DAV}.

Lastly, we mention that (strongly) David coefficients and (strongly) David maps can also be defined on open subsets $U$ of the sphere $\RiemannSphere$ provided euclidean area is replaced by spherical area in (\ref{DavidCondition}) and (\ref{StronglyDavidCondition}).

The following is a generalization of the measurable Riemann mapping theorem (cf. Theorem \ref{MRMT}) for David maps.

\begin{theorem}[David integrability theorem \cite{DAV}]
\label{DRMT}
Let $U$ be a domain in $\RiemannSphere$ and let $\mu : U \to \mathbb{D}$ be a David coefficient on $U$. Then there exists a David map $f$ on $U$ such that $\mu=\mu_f$, i.e.
$$\partial_{\overline{z}}f=\mu \, \partial_z f$$
almost everywhere on $U$. Moreover, the map $f$ is unique up to post-composition by a conformal map, in the sense that a David map $g$ on $U$ satisfies $\mu_g=\mu=\mu_f$ if and only if $f \circ g^{-1} : g(U) \to f(U)$ is conformal.
\end{theorem}

A convenient way to think of Beltrami coefficients and, more generally, of strongly David coefficients is in terms of almost complex structures and pullbacks, as in \cite[Section 1.2]{BRF}.

\begin{definition}
Let $V \subset \mathbb{C}$ be open and let $\mu:V \to \mathbb{D}$ be measurable. If $f : U \to V$ is an orientation-preserving strongly David map, then one can define a measurable function $f^{*}(\mu):U \to \mathbb{D}$, called the \textit{pullback} of $\mu$ by $f$, by
$$f^{*}(\mu) := \frac{\partial_{\overline{z}}f + (\mu \circ f) \overline{\partial_z f}}{\partial_z f + (\mu \circ f)\overline{\partial_{\overline{z}}f}}.$$
\end{definition}

Note that this is well-defined since strongly David maps are differentiable almost everywhere and they preserve sets of area zero. Furthermore, we can also consider orientation-reversing strongly David maps $f : U \to V$. In this case, the strongly David coefficient of $f$ is defined by $\mu_f := \mu_{\overline{f}}$ and the pullback $f^{*}(\mu)$ by
$$f^{*}(\mu) = \frac{\overline{\partial_{z}f} + \overline{(\mu \circ f)} \partial_{\overline{z}} f}{\overline{\partial_{\overline{z}} f} + \overline{(\mu \circ f)}\partial_{z}f}.$$

With these definitions, the coefficient $\mu_f$ is simply the pullback of $\mu_0 \equiv 0$ by $f$. Moreover, pullbacks satisfy the natural property
$$(f \circ g)^{*}(\mu) = g^*(f^*(\mu)).$$

We are now ready to discuss coefficients which are invariant under the action of the Schottky group $\Gamma(\Omega)$.

\begin{definition}
We say that a measurable function $\mu:\RiemannSphere \to \mathbb{D}$ is \textit{invariant with respect to} $\Gamma(\Omega)$ if $T^*(\mu)=\mu$ almost everywhere on $\RiemannSphere$ for every $T \in \Gamma(\Omega).$ This is equivalent to
$$\mu = (\mu \circ T) \frac{\overline{\partial_z T}}{\partial_z T}$$
or
$$\mu = \overline{(\mu \circ T)} \frac{\partial_{\overline{z}} T}{\overline{\partial_{\overline{z}}T}}$$
depending on whether $T$ is M\"{o}bius or anti-M\"{o}bius.
\end{definition}
The whole idea of trans-quasiconformal deformation of Schottky groups is based on the fact that any strongly David homeomorphism of the sphere with coefficient invariant with respect to $\Gamma(\Omega)$ maps $\Omega$ onto another circle domain.

\begin{proposition}
\label{PropInvariant}
Let $f:\RiemannSphere \to \RiemannSphere$ be a strongly David map whose strongly David coefficient $\mu_f$ is invariant with respect to the Schottky group $\Gamma(\Omega)$. Then $f(\Omega)$ is a circle domain whose corresponding Schottky group is $f \Gamma(\Omega) f^{-1}$.
\end{proposition}

\begin{proof}
To prove that $f(\Omega)$ is a circle domain, it suffices to show that $f(\gamma_j)$ is a circle, for each circle $\gamma_j$ in $\partial \Omega$. Recall that $R_j : \RiemannSphere \to \RiemannSphere$ denotes the reflection across the circle $\gamma_j$. Also, note that since strongly David maps are closed under inversions and compositions, we have that $f \circ R_j \circ f^{-1}$ is an orientation-reversing strongly David map. Moreover, if $\mu_0 \equiv 0$, then
$$(f \circ R_j \circ f^{-1})^{*}(\mu_0) = (f^{-1})^{*}(R_j^*(f^*(\mu_0))) = (f^{-1})^{*}(f^*(\mu_0)) = \mu_0,$$
 where we used the fact that $\mu_f=f^*(\mu_0)$ is invariant with respect to $\Gamma(\Omega)$. In other words, the coefficient of the strongly David map $f \circ R_j \circ f^{-1}$ is zero almost everywhere, which implies that $f \circ R_j \circ f^{-1}$ is anti-M\"{o}bius, by the uniqueness part of Theorem \ref{DRMT}. Now, note that $\gamma_j$ is the fixed point set of $R_j$, so that $f(\gamma_j)$ is the fixed point set of the anti-M\"{o}bius transformation $f \circ R_j \circ f^{-1}$. It follows that $f(\gamma_j)$ must be a circle and $f \circ R_j \circ f^{-1}$ is the reflection across this circle.

 Finally, the fact that the Schotty group of $f(\Omega)$ is $f \Gamma(\Omega) f^{-1}$ follows directly from the fact that it is generated by the family of reflections $\{f \circ R_j \circ f^{-1}\}_{j \in \mathbb{N}}$.
\end{proof}

Now, let $\Omega$ be a circle domain, let $P$ be the set of its point boundary components and set $\Omega':=\Omega \cup P$. Suppose that $\mu:\Omega' \to \mathbb{D}$ is a measurable function.

\begin{definition}
We define the \textit{invariant extension} $\tilde{\mu}$ of $\mu$ to $\RiemannSphere$ by

\begin{displaymath}
\tilde{\mu}(w) = \left\{ \begin{array}{ll}
(T^{-1})^*(\mu)(w) & \textrm{if $w \in T(\Omega')$ for some $T \in \Gamma(\Omega)$}\\
0 & \textrm{otherwise,}\\
\end{array} \right.
\end{displaymath}
for $w \in \RiemannSphere$.
\end{definition}

Note that $\tilde{\mu} = \mu$ on $\Omega'$ and that $|\tilde{\mu}(T(z))|=|\mu(z)|$ for all $z \in \Omega'$ and all $T \in \Gamma(\Omega)$. In particular, we have $\|\tilde{\mu}\|_\infty = \|\mu\|_\infty$. Moreover, by construction, the function $\tilde{\mu}$ is invariant with respect to $\Gamma(\Omega)$.

The following result of Sibner is a nice application of invariant extensions of Beltrami coefficients.

\begin{proposition}[Sibner \cite{SIB1}]
\label{QCeqCeq}
Let $D$ be a domain in $\RiemannSphere$ which is a quasiconformally equivalent to a circle domain. Then $D$ is conformally equivalent to a circle domain.
\end{proposition}

\begin{proof}
Suppose that $g:\Omega \to D$ is a quasiconformal mapping of a circle domain $\Omega$ onto $D$. Define $\mu : \Omega' \to \mathbb{D}$ by
\begin{displaymath}
\mu = \left\{ \begin{array}{ll}
\mu_g & \textrm{on $\Omega$}\\
0 & \textrm{on $P$,}\\
\end{array} \right.
\end{displaymath}
so that $\|\mu\|_\infty<1$. Then the invariant extension $\tilde{\mu}$ of $\mu$ also satisfies $\|\tilde{\mu}\|_\infty <1$. By the measurable Riemann mapping theorem, there is a quasiconformal mapping $f:\RiemannSphere \to \RiemannSphere$ with $\mu_f = \tilde{\mu}$, and since $\mu_f = \tilde{\mu}=\mu_g$ on $\Omega$, we get that $f \circ g^{-1}$ is conformal on $D$. Moreover, $(f \circ g^{-1})(D)=f(\Omega)$ is a circle domain, by Proposition \ref{PropInvariant}.
\end{proof}

\section{Proof of Theorem \ref{MainThm2}}
\label{Sec4}
In this section, we show that a circle domain is conformally rigid if and only if it is quasiconformally rigid. The precise definition of the latter is the following.

\begin{definition}
A circle domain $\Omega$ in $\RiemannSphere$ is said to be \textit{quasiconformally rigid} if every quasiconformal mapping of $\Omega$ onto another circle domain is the restriction of a quasiconformal mapping of the whole sphere.
\end{definition}

Before we proceed with the proof of Theorem \ref{MainThm2}, we need the following lemma.

\begin{lemma}
\label{ZeroArea}
Let $\Omega$ be a circle domain in $\RiemannSphere$. If $\Omega$ is quasiconformally rigid, then $\partial \Omega$ has zero area.
\end{lemma}

\begin{proof}
This is precisely where we need to use trans-quasiconformal deformation of Schottky groups. First, assume that $\infty \in \Omega$, composing with a M\"{o}bius transformation if necessary. As before, denote by $\{\gamma_j\}_{j =1}^\infty$ and $P$ respectively the collection of circles and the set of point components in $\partial \Omega$, and set $\Omega' = \Omega \cup P$. Let $\Gamma(\Omega)$ be the Schottky group of $\Omega$, which we write as $\Gamma(\Omega)=\{T_j\}_{j \geq 0}$, where $T_0$ is the identity.

Assume that the boundary of $\Omega$ has positive area, so that $m(P)>0$. Note that the sets $\{T_j(\Omega')\}_{j=1}^\infty$ are pairwise disjoint and that their union is bounded, thus we have
$$\sum_{j=1}^{\infty}m(T_j(\Omega'))<\infty.$$
For $n \in \mathbb{N}$, let $M(n) \in \mathbb{N}$ be such that
$$\sum_{j=M(n)+1}^\infty m(T_j(\Omega'))< e^{-e^n}.$$
We can assume that $M:\mathbb{N} \to \mathbb{N}$ is strictly increasing.

We now define a David coefficient as follows. For $n \in \mathbb{N}$, let
$$h(n)=\frac{e^{-e^n/2}}{\sqrt{M(n)+1}}.$$
Then $h:\mathbb{N} \to (0,\infty)$ is strictly decreasing and $h(n) \to 0$ as $n \to \infty$. Now, let $p$ be a Lebesgue density point of $P$, and define $\mu: \Omega' \to \mathbb{D}$ by

\begin{displaymath}
\mu(z) = \left\{ \begin{array}{ll}
1-\frac{1}{n+1} & \textrm{if $z \in P$ and $h(n+1) < |z-p| \leq h(n) $, $n \in \mathbb{N}$}\\
0 & \textrm{otherwise.}\\
\end{array} \right.
\end{displaymath}
Then $\mu$ satisfies the strongly David condition (\ref{StronglyDavidCondition}) on $\Omega'$. Indeed, for all $n \in \mathbb{N}$, we have
\begin{eqnarray*}
m(\{z \in \Omega' : |\mu(z)|>1-1/n\}) &=& m(P \cap \overline{\mathbb{D}}(p,h(n))) \\
&\leq& \pi h(n)^2\\
&=& \pi \frac{e^{-e^n}}{M(n)+1} \leq \pi e^{-e^{n}}.\\
\end{eqnarray*}

Now, consider the invariant extension $\tilde{\mu}$ of $\mu$. Let us check that $\tilde{\mu}$ is strongly David on $\RiemannSphere$. First, since $|\tilde{\mu}(T_j(z))|=|\mu(z)|$ for all $z \in \Omega'$ and all $j \geq 0$, we have, for $n \in \mathbb{N}$,
\begin{eqnarray*}
m(\{z \in \RiemannSphere : |\tilde{\mu}(z)| > 1-1/n\}) &=& m \left( \bigcup_{j=0}^\infty T_j\left(\{z \in \Omega': |\mu(z)|>1-1/n\}\right)\right)\\
&=& \sum_{j=0}^\infty m(T_j(\{z \in \Omega': |\mu(z)|>1-1/n\})).\\
&=& S_1+S_2,\\
\end{eqnarray*}
where
$$S_1:=\sum_{j=0}^{M(n)} m(T_j(\{z \in \Omega': |\mu(z)|>1-1/n\}))$$
and
$$S_2:=\sum_{j=M(n)+1}^\infty m(T_j(\{z \in \Omega': |\mu(z)|>1-1/n\})).$$

To estimate $S_1$, note that each $T_j$ is area-decreasing on $\Omega'$, in the sense that $m(T_j(E)) \leq m(E)$ whenever $E \subset \Omega'$ is measurable. This follows from a simple calculation involving the change of variable formula. Thus the first sum is less than
$$(M(n)+1) m(\{z \in \Omega': |\mu(z)|>1-1/n\}) \leq (M(n)+1) \pi \frac{e^{-e^n}}{M(n)+1} = \pi e^{-e^n}.$$
On the other hand, the second sum $S_2$ is less than
$$\sum_{j=M(n)+1}^{\infty} m(T_j(\Omega'))<e^{-e^n},$$
by definition of $M(n)$. Combining these two estimates together, we get, for each $n \in \mathbb{N}$,
$$m(\{z \in \RiemannSphere : |\tilde{\mu}(z)| > 1-1/n\}) \leq \pi e^{-e^n} + e^{-e^n} = (\pi+1)e^{-e^n},$$
from which it follows that $\tilde{\mu}$ is a strongly David coefficient on $\RiemannSphere$.

By Theorem \ref{DRMT}, there is a strongly David map $f:\RiemannSphere \to \RiemannSphere$ with $\mu_f = \tilde{\mu}$ almost everywhere. Since $\tilde{\mu}=\mu=0$ almost everywhere on $\Omega$, we get that $f$ is conformal and, in particular, quasiconformal on $\Omega$. Furthermore, since $\mu_f$ is invariant with respect to the Schottky group $\Gamma(\Omega)$, it follows from Proposition \ref{PropInvariant} that $f(\Omega)$ is a circle domain.

Finally, the map $f$ on $\Omega$ is clearly not the restriction of a quasiconformal mapping of the whole sphere. Indeed, this easily follows from the fact that for any $0 \leq k <1$, we can find some integer $n \in \mathbb{N}$ such that $1-1/(n+1)>k$, but then
$$|\mu_f| = |\tilde{\mu}|=|\mu| \geq 1-\frac{1}{n+1} > k$$
almost everywhere on $P \cap \overline{\mathbb{D}}(p,h(n))$, a set of positive area, since $p$ is a Lebesgue density point of $P$.

This shows that $\Omega$ is not quasiconformally rigid.

\end{proof}

We also need the following lemma.

\begin{lemma}[He--Schramm]
\label{LemmaHS}
Let $\Omega$ be a circle domain, let $P$ be the set of its point boundary components and let $\Omega'=\Omega \cup P$. Suppose that $g:\RiemannSphere \to \RiemannSphere$ is a quasiconformal mapping which maps $\Omega$ onto a circle domain $g(\Omega)$. Let $\mu$ be the Beltrami coefficient of $g$ restricted to $\Omega'$, and denote by $\tilde{\mu}$ its invariant extension to $\RiemannSphere$. Suppose that $h:\RiemannSphere \to \RiemannSphere$ is a quasiconformal mapping with $\mu_h = \tilde{\mu}$ almost everywhere. Then $g=T \circ h$ on $\Omega$ for some M\"{o}bius transformation $T$.
\end{lemma}

\begin{proof}
The proof follows from classical results of Sullivan on Kleinian groups, see \cite[Lemma 2.2]{SCH4}.

\end{proof}

We can now proceed with the proof of Theorem \ref{MainThm2}.

\begin{proof}
Let $\Omega$ be a circle domain. We want to show that $\Omega$ is conformally rigid if and only if it is quasiconformally rigid.

Assume that $\Omega$ is conformally rigid, and let $f: \Omega \to f(\Omega)$ be quasiconformal, where $f(\Omega)$ is a circle domain. Set $\mu:=\mu_{f^{-1}}$ on $f(\Omega)$ and consider the invariant extension $\tilde{\mu}$ of $\mu$. By the measurable Riemann mapping theorem, there is a quasiconformal mapping $g:\RiemannSphere \to \RiemannSphere$ with $\mu_g=\tilde{\mu}$ almost everywhere on $\RiemannSphere$. Then $g \circ f$ is conformal on $\Omega$ and $g(f(\Omega))$ is a circle domain, by Proposition \ref{PropInvariant}. Since $\Omega$ is conformally rigid, it follows that $g \circ f$ is the restriction of a M\"{o}bius transformation and thus $f= g^{-1} \circ (g \circ f)$ is the restriction of a quasiconformal mapping of the whole sphere. This shows that $\Omega$ is quasiconformally rigid.

For the converse, assume that $\Omega$ is quasiconformally rigid, and let $f:\Omega \to f(\Omega)$ be conformal, where $f(\Omega)$ is a circle domain. Then in particular $f$ is quasiconformal on $\Omega$, hence is the restriction of a quasiconformal mapping $g$ of the whole sphere, by quasiconformal rigidity of $\Omega$. Now, by Lemma \ref{ZeroArea}, the boundary $\partial \Omega$ has zero area and thus the Beltrami coefficient of $g$ restricted to $\Omega'$ is zero almost everywhere. It follows that the map $h:\RiemannSphere \to \RiemannSphere$ of Lemma \ref{LemmaHS} is a M\"{o}bius transformation. Since $f=g=T \circ h$ on $\Omega$, we get that $f$ is the restriction of a M\"{o}bius transformation. This shows that $\Omega$ is conformally rigid.

\end{proof}

As mentioned in the introduction, it follows that rigid circle domains are quasiconformally invariant.

\begin{cor}
Let $\Omega$ be a circle domain and let $f$ be a quasiconformal mapping of the sphere which maps $\Omega$ onto another circle domain $f(\Omega)$. If $\Omega$ is conformally rigid, then $f(\Omega)$ is also conformally rigid.
\end{cor}

\section{Concluding remarks}
\label{Sec5}

We conclude the paper with a few remarks regarding the equivalence of \textbf{(A)} and \textbf{(B)} in Conjecture \ref{RigidityConjecture}.

Assume for simplicity that $\Omega$ is the complement in $\RiemannSphere$ of some Cantor set $E \subset \mathbb{C}$. Then $\Omega$ is a circle domain, and recall from the introduction that $E$ is conformally removable whenever $\Omega$ is conformally rigid; in other words, \textbf{(A)} implies \textbf{(B)}. It is not known whether the converse holds, even in this special case.

Suppose that $E$ is a counterexample to the converse, so that $E$ is conformally removable but $\Omega$ is not conformally rigid. Then in particular $E$ cannot be removable for conformal maps on its complement, in the sense that there exists a non-M\"{o}bius conformal map on $\Omega$. Examples of Cantor sets which are conformally removable but not removable for conformal maps on their complement were given by Ahlfors and Beurling in their seminal paper \cite{AHB}. More precisely, they proved that if $E$ is a Cantor subset of the unit circle $\mathbb{T}$ and if the inner logarithmic capacity of $\mathbb{T} \setminus E$ is less than one, then $E$ is not removable for conformal maps on $\Omega=\RiemannSphere \setminus E$ (see \cite[Theorem 14]{AHB}). Note that such compact sets $E$ are conformally removable by Morera's theorem. On the other hand, they all have finite length and thus their complements $\Omega$ must be conformally rigid, by \cite{SCH2}.

In this respect, it would be very interesting for the study of rigidity of circle domains to find examples of Cantor sets $E$ which

\begin{enumerate}[\rm(i)]
\item do not have $\sigma$-finite length;
\item are conformally removable;
\item are not removable for conformal maps on their complement
\end{enumerate}
and are fundamentally different from the Cantor sets of Ahlfors and Beurling.

Finally, note that if $E$ is as above, then there is a non-M\"{o}bius conformal map $f$ on $\Omega=\RiemannSphere \setminus E$, and $f$ does not extend to a homeomorphism of $\RiemannSphere$. It follows that $f$ must stretch some points of $E$ to nondegenerate continuum. The question is then whether it is possible for all of these continuum to be circles.

\bibliographystyle{amsplain}

\end{document}